\documentclass[10pt,reqno]{amsart}
\usepackage{amsmath}
\usepackage{amssymb}
\usepackage{amsthm}

\newlength{\defbaselineskip}
\newcommand{\setlinespacing}[1]%
           {\setlength{\baselineskip}{#1 \defbaselineskip}}

\numberwithin{equation}{section}

\newtheorem{thm}{Theorem}[section]

\newtheorem{lem}[thm]{Lemma}

\theoremstyle{definition}

\theoremstyle{remark}

\numberwithin{equation}{section}

\begin{document}

\title[Unique continuation for the Cauchy-Riemann operator]
{A remark on unique continuation for the Cauchy-Riemann operator}

\author{Ihyeok Seo}


\address{Department of Mathematics, Sungkyunkwan University, Suwon 440-746, Republic of Korea}
\email{ihseo@skku.edu}


\begin{abstract}
In this note we obtain a unique continuation result for the differential inequality
$|\overline{\partial}u|\leq|Vu|$, where $\overline{\partial}=(i\partial_y+\partial_x)/2$
denotes the Cauchy-Riemann operator and $V(x,y)$ is a function in $L^2(\mathbb{R}^2)$.
\end{abstract}

\maketitle


\section{Introduction}

The unique continuation property is one of the most interesting properties of holomorphic functions $f\in H(\mathbb{C})$.
This property says that if $f$ vanishes in a non-empty open subset
of $\mathbb{C}$ then it must be identically zero.
Note that $u\in C^1(\mathbb{R}^2)$ satisfies the Cauchy-Riemann equation $(i\partial_y+\partial_x)u=0$
if and only if it defines a holomorphic function $f(x+iy)\equiv u(x,y)$ on $\mathbb{C}$.
From this point of view, one can see that a $C^1$ function satisfying the equation has the unique continuation property.

In this note we consider a class of non-holomorphic functions $u$ which satisfy
the differential inequality
\begin{equation}\label{inequality}
|\overline{\partial}u|\leq|Vu|,
\end{equation}
where $\overline{\partial}=(i\partial_y+\partial_x)/2$ denotes the Cauchy-Riemann operator
and $V(x,y)$ is a function on $\mathbb{R}^2$.

The best positive result for \eqref{inequality} is due to Wolff \cite{W}
(see Theorem 4 there) who proved the property for $V\in L^p$ with $p>2$.
On the other hand, there is a counterexample \cite{M} to unique continuation for \eqref{inequality}
with $V\in L^p$ for $p<2$.
The remaining case $p=2$ seems to be unknown for the differential inequality \eqref{inequality},
and note that $L^2$ is a scale-invariant space of $V$ for the equation $\overline{\partial}u=Vu$.
Here we shall handle this problem.
Our unique continuation result is the following theorem
which is based on bounds for a Fourier multiplier from $L^p$ to $L^q$.

\begin{thm}\label{thm}
Let $1<p<2<q<\infty$ and $1/p-1/q=1/2$.
Assume that $u\in L^p\cap L^q$ satisfies the inequality \eqref{inequality} with $V\in L^2$
and vanishes in a non-empty open subset of $\mathbb{R}^2$.
Then it must be identically zero.
\end{thm}

The unique continuation property also holds for harmonic functions,
which satisfy the Laplace equation $\Delta u=0$,
since they are real parts of holomorphic functions.
This was first extended by Carleman ~\cite{C} to a class of non-harmonic functions
satisfying the inequality $|\Delta u|\leq|Vu|$ with $V\in L^\infty(\mathbb{R}^2)$.
There is an extensive literature on later developments in this subject.
In particular, the problem of finding all the possible $L^p$ functions $V$,
for which $|\Delta u|\leq|Vu|$ has the unique continuation, is completely solved (see \cite{JK,KN,KT}).
See also the survey papers of Kenig \cite{K} and Wolff \cite{W2} for more details,
and the recent paper of Kenig and Wang \cite{KW} for a stronger result
which gives a quantitative form of the unique continuation.

\medskip

Throughout the paper, the letter $C$ stands for positive constants possibly different
at each occurrence.
Also, the notations $\widehat{f}$ and $\mathcal{F}^{-1}(f)$ denote
the Fourier and the inverse Fourier transforms of $f$, respectively.


\section{A preliminary lemma}

The standard method to study the unique continuation property
is to obtain a suitable Carleman inequality for relevant differential operator.
This method originated from Carleman's classical work \cite{C} for elliptic operators.
In our case we need to obtain the following inequality for the Cauchy-Riemann operator $\overline{\partial}=(i\partial_y+\partial_x)/2$,
which will be used in the next section for the proof of Theorem \ref{thm}:

\begin{lem}\label{lem}
Let $f\in C_0^\infty(\mathbb{R}^2\setminus\{0\})$.
For all $t>0$, we have
\begin{equation}\label{Sobo}
\big\||z|^{-t}f\big\|_{L^q}
\leq C\big\||z|^{-t}\overline{\partial}f\big\|_{L^p}
\end{equation}
if $1<p<2<q<\infty$ and $1/p-1/q=1/2$.
Here, $z=x+iy\in\mathbb{C}$ and $C$ is a constant independent of $t$.
\end{lem}

\begin{proof}
First we note that
$$\overline{\partial}(z^{-t}f)=z^{-t}\overline{\partial}f+f\overline{\partial}(z^{-t})
=z^{-t}\overline{\partial}f$$
for $z\in\mathbb{C}\setminus\{0\}$.
Then the inequality \eqref{Sobo} is equivalent to
$$\big\|z^{-t}f\big\|_{L^q}
\leq C\big\|\overline{\partial}(z^{-t}f)\big\|_{L^p}.$$
By setting $g=z^{-t}f$, we are reduced to showing that
\begin{equation*}
\|g\|_{L^q}
\leq C\|(i\partial_y+\partial_x)g\|_{L^p}
\end{equation*}
for $g\in C_0^\infty(\mathbb{R}^2\setminus\{0\})$.
To show this, let us first set
\begin{equation}\label{0}
(i\partial_y+\partial_x)g=h,
\end{equation}
and let $\psi_\delta:\mathbb{R}^2\rightarrow[0,1]$ be a smooth function
such that $\psi_\delta=0$ in the ball $B(0,\delta)$
and $\psi_\delta=1$ in $\mathbb{R}^2\setminus B(0,2\delta)$.
Then, using the Fourier transform in~\eqref{0}, we see that
\begin{equation*}
(-\eta+i\xi)\widehat{g}(\xi,\eta)=\widehat{h}(\xi,\eta).
\end{equation*}
Thus, by Fatou's lemma we are finally reduced to showing the following uniform boundedness
for a multiplier operator having the multiplier
$m(\xi,\eta)=\psi_\delta(\xi,\eta)/(-\eta+i\xi)$:
\begin{equation}\label{multi}
\bigg\|\mathcal{F}^{-1}\bigg(\frac{\psi_\delta(\xi,\eta)
\widehat{h}(\xi,\eta)}{-\eta+i\xi}\bigg)\bigg\|_{L^q}\leq C\|h\|_{L^p}
\end{equation}
uniformly in $\delta>0$.

From now on, we will show \eqref{multi} using Young's inequality for convolutions and Littlewood-Paley theorem (\cite{G}).
Let us first set for $k\in\mathbb{Z}$
$$\widehat{Th}(\xi,\eta)=m(\xi,\eta)\widehat{h}(\xi,\eta)\quad\text{and}\quad
\widehat{T_kh}(\xi,\eta)=m(\xi,\eta)\chi_k(\xi,\eta)\widehat{h}(\xi,\eta),$$
where $\chi_k(\cdot)=\chi(2^k\cdot)$ for $\chi\in C_0^\infty(\mathbb{R}^2)$ which is such that
$\chi(\xi,\eta)=1$ if $|(\xi,\eta)|\sim1$, and zero otherwise.
Also, $\sum_k\chi_k=1$.
Now we claim that
\begin{equation}\label{multi2}
\|T_kh\|_{L^q}\leq C\|h\|_{L^p}
\end{equation}
uniformly in $k\in\mathbb{Z}$.
Then, since $1<p<2<q<\infty$, by the Littlewood-Paley theorem together with Minkowski's inequality,
we get the desired inequality \eqref{multi} as follows:
\begin{align*}
\big\|\sum_kT_kh\big\|_{L^q}&\leq C\big\|\big(\sum_k|T_kh|^2\big)^{1/2}\big\|_{L^q}\\
&\leq C\big(\sum_k\|T_kh\|_{L^q}^2\big)^{1/2}\\
&\leq C\big(\sum_k\|h_k\|_{L^p}^2\big)^{1/2}\\
&\leq C\big\|\big(\sum_k|h_k|^2\big)^{1/2}\big\|_{L^p}\\
&\leq C\big\|\sum_kh_k\big\|_{L^p},
\end{align*}
where $h_k$ is given by $\widehat{h_k}(\xi,\eta)=\chi_k(\xi,\eta)\widehat{h}(\xi,\eta)$.
Now it remains to show the claim \eqref{multi2}.
But, this follows easily from Young's inequality.
Indeed, note that
$$T_kh=\mathcal{F}^{-1}\bigg(\frac{\psi_\delta(\xi,\eta)\chi_k(\xi,\eta)}{-\eta+i\xi}\bigg)
\ast h$$
and by Plancherel's theorem
\begin{align*}
\bigg\|\mathcal{F}^{-1}\bigg(\frac{\psi_\delta(\xi,\eta)\chi_k(\xi,\eta)}{-\eta+i\xi}\bigg)\bigg\|_{L^2}
&=\bigg\|\frac{\psi_\delta(\xi,\eta)\chi_k(\xi,\eta)}{-\eta+i\xi}\bigg\|_{L^2}\\
&\leq C\bigg(\int_{|(\xi,\eta)|\sim2^{-k}}\frac1{\eta^2+\xi^2}d\xi d\eta\bigg)^{1/2}\\
&\leq C.
\end{align*}
Since we are assuming the gap condition $1/p-1/q=1/2$,
by Young's inequality for convolutions,
this readily implies that
$$\|T_kh\|_{L^q}\leq \bigg\|\mathcal{F}^{-1}\bigg(\frac{\psi_\delta(\xi,\eta)\chi_k(\xi,\eta)}{-\eta+i\xi}\bigg)\bigg\|_{L^2}\|h\|_{L^p}
\leq C\|h\|_{L^p}$$
as desired.
\end{proof}

\section{Proof of Theorem \ref{thm}}
The proof is standard once one has the Carleman inequality \eqref{Sobo} in Lemma \ref{lem}.

Without loss of generality, we may show that $u$ must vanish identically
if it vanishes in a sufficiently small neighborhood of zero.
Then, since we are assuming that $u\in L^p\cap L^q$ vanishes near zero,
from \eqref{Sobo} with a standard limiting argument involving a $C_0^\infty$ approximate identity,
it follows that
\begin{equation*}
\big\||z|^{-t}u\big\|_{L^q}
\leq C\big\||z|^{-t}\overline{\partial}u\big\|_{L^p}.
\end{equation*}
Thus by \eqref{inequality} we see that
\begin{align*}
\big\||z|^{-t}u\big\|_{L^q(B(0,r))}&\leq C\big\||z|^{-t}Vu\big\|_{L^p(B(0,r))}\\
&+C\big\||z|^{-t}\overline{\partial}u\big\|_{L^p(\mathbb{R}^2\setminus B(0,r))},
\end{align*}
where $B(0,r)$ is the ball of radius $r>0$ centered at $0$. 
Then, using H\"older's inequality with $1/p-1/q=1/2$,
the first term on the right-hand side in the above can be absorbed into the left-hand side as follows:
\begin{align*}
C\big\||z|^{-t}Vu\big\|_{L^p(B(0,r))}&\leq
C\|V\|_{L^2(B(0,r))}\big\||z|^{-t}u\big\|_{L^q(B(0,r))}\\
&\leq \frac12\big\||z|^{-t}u\big\|_{L^q(B(0,r))}
\end{align*}
if we choose $r$ small enough.
Here, $\||z|^{-t}u\|_{L^q(B(0,r))}$ is finite since $u\in L^q$ vanishes near zero.
Hence we get
\begin{align*}
\|(r/|z|)^{t}u\|_{L^q(B(0,r))}
&\leq2C\|\overline{\partial}u\|_{L^p(\mathbb{R}^2\setminus B(0,r))}\\
&\leq2C\|V\|_{L^2}\|u\|_{L^q}\\
&<\infty.
\end{align*}
By letting $t\rightarrow\infty$, we now conclude that $u=0$ on $B(0,r)$.
This implies $u\equiv0$ by a standard connectedness argument.

\section*{Acknowledgment}
I would like to thank Jenn-Nan Wang for pointing out a preprint (\cite{KW}) and for some comments.

\bibliographystyle{plain}


\end{document}